\documentclass[11pt, reqno]{amsart}
\usepackage[dvipsnames,usenames]{color}
\usepackage[colorlinks=true, urlcolor=NavyBlue, linkcolor=NavyBlue, citecolor=NavyBlue]{hyperref}
\usepackage{cleveref}
\usepackage{graphicx}
\usepackage{epsfig}
\usepackage[latin1]{inputenc}
\usepackage{amsmath}
\usepackage{amsfonts}
\usepackage{amssymb}
\usepackage{amsthm}
\usepackage{amscd}
\usepackage{verbatim}
\usepackage{subfigure}
\usepackage{caption}
\usepackage{pinlabel}
\usepackage{stmaryrd}
\usepackage{enumerate, enumitem}
\usepackage{todonotes}
\usepackage{bm}
\usepackage{thmtools}
\usepackage{thm-restate}
\usepackage{lipsum}
\usepackage{setspace}
\usepackage{mathtools}
\usepackage[all]{xypic}
\usepackage[abs]{overpic}
\usepackage{color}
\usepackage[normalem]{ulem}
\usepackage[alphabetic,backrefs,msc-links]{amsrefs}

\allowdisplaybreaks

\usepackage{tikz}
\usetikzlibrary{arrows}
\usetikzlibrary{decorations.pathreplacing}
\usepackage{verbatim}
\usetikzlibrary{cd}
\tikzset{taar/.style={double, double equal sign distance, -implies}}
\tikzset{amar/.style={->, dotted}}
\tikzset{dmar/.style={->, dashed}}
\tikzset{aar/.style={->, very thick}}

\newtheorem{theorem}{Theorem}[section]

\newtheorem{lemma}[theorem]{Lemma}
\newtheorem{proposition}[theorem]{Proposition}

\newtheorem{corollary}[theorem]{Corollary}

\newtheorem{question}[theorem]{Question}

\theoremstyle{definition}

\theoremstyle{remark}
\newtheorem{remark}[theorem]{Remark}
\newtheorem{example}[theorem]{Example}

\def\F{\mathbb{F}}

\def\d{\partial}

\def\CF {\mathit{CF}}

\newcommand \CFm {\CF^-}

\def\CFK{\mathit{CFK}}

\newcommand{\lab}[1]{$\scriptstyle #1$}

\author[J.\ Hom]{Jennifer Hom}
\address {School of Mathematics, Georgia Institute of Technology, Atlanta, GA 30332}
\email{hom@math.gatech.edu}

\author[J.\ Park]{JungHwan Park}
\address{Department of Mathematical Sciences, KAIST, Daejeon, South Korea}
\email{jungpark0817@kaist.ac.kr}

\numberwithin{equation}{section}

\title{Ribbon knots and iterated cables of fibered knots}

\begin{document}

\begin{abstract}
We define a knot to be \emph{$\gamma_0$-sharp} if its Seifert genus is detected by the concordance invariant $\gamma_0$, which arises from the immersed curve formalism in bordered Heegaard Floer homology. We show that a connected sum of $\gamma_0$-sharp fibered knots is ribbon exactly when it is of the form $K \mathbin{\#} -K$.  Consequently, either iterated cables of tight fibered knots are linearly independent in the smooth concordance group, or the slice--ribbon conjecture fails.
\end{abstract}

\maketitle

\section{Introduction}
A knot in the three-sphere $S^3$ is called \emph{slice} if it bounds a smoothly embedded disk in the four-ball $B^4$, and \emph{ribbon} if it bounds an immersed disk in $S^3$ with only ribbon singularities. Since resolving the ribbon singularities naturally yields a smooth slicing disk, every ribbon knot is necessarily slice. One of the most significant open problems in low-dimensional topology, originally posed by Fox~{\cite[Problem 25]{Fox:1961}}, asks whether the converse is also true. This question is known as the \emph{slice-ribbon conjecture}. 

By applying the celebrated theorem of Casson and Gordon~\cite{Casson-Gordon:1983}, which states that a fibered knot in a homology sphere is homotopically ribbon\footnote{A knot in a homology sphere is called \emph{homotopically ribbon} if it bounds a smooth disk in a homology four-ball, where the inclusion of the knot complement into the disk complement induces a surjective map on the fundamental group. Note that every ribbon knot in the three-sphere is homotopically ribbon.}  if and only if its closed monodromy extends over a handlebody, many interesting non-ribbon knots have been constructed. For a large number of them, the sliceness problem remains unresolved to this day.

For instance, Miyazaki~\cite{Miyazaki:1994} proved that a connected sum of iterated cables of torus knots is not ribbon unless it is of the form $K \mathbin{\#} -K$, where $K$ is a knot and $-K$ denotes its mirror image with reversed orientation. In contrast, the sliceness problem for connected sums of iterated cables of torus knots remains widely open. In particular, a well-known and long-standing conjecture of Rudolph~\cite{Rudolph:1976} states that the set of algebraic knots, which are in particular iterated cables of torus knots, is linearly independent in the smooth knot concordance group.

About ten years ago, Miyazaki's result was strengthened by Baker~\cite{Baker:2016}, who proved that the same holds for connected sums of fibered knots supporting the tight contact structure. For simplicity, we will refer to such knots as \emph{tight fibered knots} throughout this article. This led to a stronger version of Rudolph's conjecture, which states that prime tight fibered knots are linearly independent in the smooth knot concordance group (see~\cite[Corollary 5]{Baker:2016} and \cite[Lemma 3.1]{Abe-Tagami:2016}).

In this article, we provide additional examples of fibered knots that are not ribbon, using the results of Casson and Gordon~\cite{Casson-Gordon:1983} and Miyazaki~\cite{Miyazaki:1994}, combined with the concordance invariant $\gamma_0$, which arises from the interpretation of bordered Floer homology \cite{LOT} via immersed curves by Hanselman, Rasmussen, and Watson~\cite{HRW:2022, HRW:2024, Hanselman-Watson:2023}. We define a knot $K$ to be \emph{$\gamma_0$-sharp} if its Seifert genus agrees with the maximal Alexander grading of the summand of the knot Floer complex associated to $\gamma_0(K)$; that is, the Seifert genus is determined by $\gamma_0$. In particular, every tight fibered knot is $\gamma_0$-sharp (see Remark~\ref{rmk:tightfiberedgammashparp}).

%

\begin{theorem}\label{thm:main}
A connected sum of $\gamma_0$-sharp fibered knots  is ribbon if and only if it is of the form $K \mathbin{\#} -K$.
\end{theorem}



For the proof, we revisit a key step in~\cite[Lemma 2]{Baker:2016}, where Baker establishes the minimality of tight fibered knots with respect to homotopy ribbon concordance among all fibered knots in $S^3$, and we extend this result to $\gamma_0$-sharp knots.  
Here, given two knots $K_0 \subset Y_0$ and $K_1 \subset Y_1$, where $Y_0$ and $Y_1$ are homology spheres, we say that $K_1$ is \emph{homotopy ribbon concordant} to $K_0$ if there exists a smooth annulus $A$ in a homology cobordism $X$ from $Y_0$ to $Y_1$ such that
\[
(\partial X, A \cap \partial X) = (Y_1, K_1) \sqcup -(Y_0, K_0),
\]
and the inclusion-induced maps
\[
\pi_1(Y_0 \smallsetminus K_0) \hookrightarrow \pi_1(X \smallsetminus A) \qquad \text{and} \qquad \pi_1(Y_1 \smallsetminus K_1) \twoheadrightarrow \pi_1(X \smallsetminus A)
\]
are injective and surjective, respectively.
Note that homotopy ribbon concordance generalizes the notion of \emph{ribbon concordance}, introduced by Gordon~\cite[Lemma 3.1]{Gordon:1981}, which was later shown by Agol~\cite[Theorem 1.1]{Agol:2022} to define a partial order.

Furthermore, we prove that $\gamma_0$-sharpness is preserved under satelliting by 1-bridge braided patterns, which include cabling (see Proposition~\ref{prop:preservedcabling}). In particular, iterated cables of tight fibered knots remain $\gamma_0$-sharp, even though they are no longer tight fibered when the cabling parameter is negative. As a consequence, combining this fact with the main theorem yields a new family of non-ribbon knots:


\begin{corollary}\label{cor:main}
Either distinct iterated cables of tight fibered knots are linearly independent in the smooth knot concordance group, or the slice-ribbon conjecture is false.
\end{corollary}

In the following corollary, we present a special case of Corollary~\ref{cor:main}, which provides many simple and interesting new examples of non-ribbon knots whose sliceness is not easily detected.  
Notably, all of these knots are algebraically slice. Throughout this article, we use the notation $K_{p,q}$ to denote the $(p, q)$-cable of a knot $K$, where $p$ indicates the longitudinal winding. We assume $p > 1$, since $K_{p,q}$ is isotopic to $K_{-p,-q}$ and $K_{1,q}$ is isotopic to $K$.

\begin{corollary}\label{cor:examples}
Let $K$ and $J$ be distinct iterated cables of tight fibered knots. Then, for each pair of distinct integers $q_1$ and $q_2$ that are relatively prime to $p$, the knot
\[
P(K, J, p, q_1, q_2) := K_{p,q_1} \mathbin{\#} -K_{p,q_2} \mathbin{\#} J_{p,q_2} \mathbin{\#} -J_{p,q_1}
\]
is algebraically slice but not ribbon.
\end{corollary}

We make two remarks on the corollary. 
First, we discuss the algebraic sliceness and non-ribbonness of the knots in question. The fact that the knot is algebraically slice follows from work of Livingston and Melvin~\cite{Livingston-Melvin:1983} (see also~\cite{Kearton:1979}), and holds for any choice of knots $K$ and $J$ and any integers $p$,  $q_1$ and $q_2$ with $q_1$ and $q_2$ relatively prime to $p$.
The non-ribbonness of $P(K, J, p, q_1, q_2)$ follows from work of Miyazaki~\cite{Miyazaki:1994} in the case where $K$ and $J$ are distinct iterated cables of torus knots, and from work of Baker~\cite{Baker:2016} when $K$ and $J$ are distinct tight fibered knots and $q_1, q_2 > 0$, as each cable remains a tight fibered knot under these conditions.
As mentioned earlier, our contribution lies in extending these results to the case where $K$ and $J$ are arbitrary iterated cables of tight fibered knots, and in particular, the parameters $q_1$ and $q_2$ are allowed to be negative integers.

Secondly, while the slice--ribbon conjecture predicts that each knot from Corollary~\ref{cor:examples} is not slice, proving this remains challenging in general with current techniques in knot concordance. Nonetheless, notable results by Hedden, Kirk, and Livingston~\cite{Hedden-Kirk-Livingston:2012}, Conway, Kim, and Politarczyk~\cite{Conway-Kim-Politarczyk:2023}, and Zhang~\cite{Zhang} prove that certain families of such knots can indeed be shown to be non-slice.

More precisely, sliceness is obstructed in~\cite[Theorem 3]{Hedden-Kirk-Livingston:2012} using Casson--Gordon theory~\cite{Casson-Gordon:1986} in the case where $K$ is the torus knot $T_{2,3}$, $J$ is the unknot, $p = 2$, and $q_1$ and $q_2$ are chosen appropriately. This result was later generalized by Zhang~\cite[Theorem 2]{Zhang}, using the Ozsv\'{a}th--Szab\'{o} $d$-invariant~\cite{Oz-Sz:dinvt} of the double branched covers, to the case where $K$ is a nontrivial $L$-space knot, $J$ is the unknot, $p = 2$, and $q_1$ and $q_2$ are distinct prime integers greater than $3$.
In addition,~\cite[Theorem 1.1]{Conway-Kim-Politarczyk:2023}, using metabelian twisted Blanchfield pairings~\cite{Miller-Powell:2018,BCP:2022-1,BCP:2018-1,BCP:2024-3}, establishes non-sliceness when $K$ and $J$ are certain positively iterated torus knots, and $p$, $q_1$, and $q_2$ are suitable positive integers.
Furthermore, when $q_1$ and $q_2$ have opposite signs, one can apply 
the cabling formula~\cite{Hom:2014}  for the Ozsv\'{a}th--Szab\'{o} $\tau$-invariant~\cite{Oz-Sz:2003}, derived from bordered Heegaard Floer homology, to conclude that the corresponding knot is not slice.

However, in general, many examples from Corollary~\ref{cor:examples} remain undetectable by any known concordance invariants.  
For instance, it was verified in~\cite[Proposition 8.2]{Hedden-Kirk-Livingston:2012} that $P(K, U, p, q_1, q_2)$ has vanishing Ozsv\'{a}th--Szab\'{o} $\tau$-invariant~\cite{Oz-Sz:2003} and Rasmussen's $s$-invariant~\cite{Rasmussen:2010} when $K$ is a positively iterated torus knot.  
Moreover,~\cite[Section 1.2]{Conway-Kim-Politarczyk:2023} shows that the $\Upsilon$-invariant of Ozsv\'{a}th--Stipsicz--Szab\'{o}~\cite{OSS:2017} also vanishes for many such examples.  
In fact, we unify and generalize these observations by proving that, for a broad class of knots, all concordance invariants derived from the knot Floer complex~\cite{OS-knots, Rasmussen-thesis} and the involutive knot Floer complex~\cite{HendricksManolescu} vanish, an observation that may be of independent interest.


\begin{proposition}\label{prop:examples}
Let $K$ be a nontrivial L-space knot, and let $U$ denote the unknot. Then, for each pair of distinct odd integers $q_1$ and $q_2$, the knot
\[
P(K, U, 2, q_1, q_2) := K_{2,q_1} \mathbin{\#} -K_{2,q_2} \mathbin{\#} T_{2,q_2} \mathbin{\#} -T_{2,q_1}
\]
is algebraically slice but not ribbon. Moreover, its knot Floer complex  is locally equivalent to that of the unknot if and only if the pair $(q_1, q_2)$ satisfies one of the following:
\begin{itemize}
    \item $q_1, q_2 > 4g(K)$,
    \item $0 < q_1, q_2 < 4g(K)$, or
    \item $q_1, q_2 < 0$,
\end{itemize}
where $g(K)$ is the Seifert genus of $K$. In particular, in all of the above cases, all concordance invariants derived from the knot Floer complex vanish.

Furthermore, when $q_1, q_2 > 4g(K)$, the obstruction coming from the involutive knot Floer complex vanishes. When $q_1, q_2 < 0$, the obstruction coming from the Ozsv\'{a}th--Szab\'{o} $d$-invariant, together with the Hendricks--Manolescu $\overline{d}$- and $\underline{d}$-invariants of the double branched covers, vanish.
\end{proposition}

Recall that every L-space knot is a tight fibered knot~\cite{OS:fibered, Ghiggini:2008, Ni:2007, Hedden:2010}, yet the class of L-space knots is much broader than that of torus knots (see, for example,~\cite{OS:lensspace, Saito-Teragaito:2010}). Since the proposition above provides non-ribbon knots beyond those considered in~\cite{Miyazaki:1994, Baker:2016}, this naturally leads to the following question, simply by taking $q_2 = -1$:

\begin{question}
Is there a nontrivial L-space knot $K$ and an integer $q<-1$ such that $K_{2,q}$ is concordant to $T_{2,q} \mathbin{\#} K_{2,-1}$?
\end{question}

\subsection*{Organization}
In Section~\ref{sec:background}, we review the knot Floer complex and the concordance invariant~$\gamma_0$. In Section~\ref{sec:mainthm}, we present the proof of the main theorem and its corollaries. Finally, in Section~\ref{sec:computations}, we prove Proposition~\ref{prop:examples}.

\subsection*{Acknowledgements} 
We would like to thank Jaewon Lee for helpful conversations. The first author was partially supported by NSF grants DMS-2104144 and DMS-2506400, and Georgia Tech's Elaine M. Hubbard Distinguished Faculty Award.
The second author is partially supported by the Samsung Science and Technology Foundation (SSTF-BA2102-02) and the NRF grant RS-2025-00542968.

\section{Background on knot Floer complexes and immersed curves}\label{sec:background}

We assume the reader is familiar with the knot Floer homology package of \cite{OS-knots} and \cite{Rasmussen-thesis}; for survey articles on these invariants, we recommend \cite{OS-intro}, \cite{Hom-survey}, and \cite{Hom-lecture}. We use the conventions of \cite{Zemke-link} and \cite{DHST-conc}, and view the knot Floer complex $\CFK(K)$ as a chain complex over the ring $\F[U, V]$. The knot Floer complex can be endowed with the additional structure of an endomorphism $\iota_K$, coming from involutive Heegaard Floer homology \cite{HendricksManolescu}.

If two knots $K$ and $J$ are concordant, then their knot Floer complexes $\CFK(K)$ and $\CFK(J)$ are \emph{locally equivalent}: there exist graded chain maps
\[
f \colon \CFK(K) \to \CFK(J) \qquad \textup{ and } \qquad g \colon \CFK(J) \to \CFK(K)
\]
that induce isomorphisms in $U^{-1}H_*$.  
We will also say that two knots $K$ and $J$ are \emph{locally equivalent} if their knot Floer complexes are locally equivalent.

Local equivalence can be considered in both weaker and stronger settings. The weaker setting is to quotient by the ideal $(UV)$, and consider the knot Floer complex over $\F[U,V]/(UV)$, as in \cite[Definition 2.3]{DHST-conc}; this is the same as $\varepsilon$-equivalence, considered in \cite[Section 4.1]{Hom-CFKconc}. The stronger setting is to require that the maps $f$ and $g$ homotopy-commute with $\iota_K$, as in~\cite[Definition~2.4]{Zemke-connected}, which we will refer to as \emph{$\iota_K$-local equivalence}.

The local equivalence class of $\CFK(K)$ over $R=\F[U,V]$ or $\F[U,V]/(UV)$ admits a unique minimal (with respect to rank) representative, called the \emph{connected complex over $R$}, which has antecedents in the Pin(2)-monopole setting \cite[Section 2.5]{Stoffregen-Pin2} and the involutive Heegaard Floer setting \cite[Section 3]{HHL}; see also \cite[Section 6]{Zhou-homconc} and \cite[Section 5]{HKPS:2022}.

As observed in \cite[Proposition 2]{Hanselman-Watson:2023}, the connected complex over $\F[U,V]/(UV)$ is exactly the data of the component $\gamma_0$ coming from the immersed curve associated to $K$ \cite{HRW:2024}. Moreover, following \cite[Section 4]{DHST-conc}, the connected complex over $\F[U,V]/(UV)$, or equivalently $\gamma_0$, can be completely described by a finite sequence of nonzero integers, specifying its standard complex representative. In Section~\ref{sec:computations}, we will say that a knot $K$ is locally equivalent to a sequence if its complex $\CFK(K)$ is locally equivalent to the standard complex associated to that sequence.

\begin{remark}\label{rmk:homologyconcordance}
The connected complex is a concordance invariant and in fact is an invariant of concordances in homology cobordisms \cite{DHST-homconc}. It follows from the proof of \cite[Proposition 2]{Hanselman-Watson:2023} that $\gamma_0$ is an invariant of concordances in homology cobordisms as well.
\end{remark}

Recall from~\cite{OS-genus} that the knot Floer complex~$\CFK(K)$ detects the genus of~$K$, in that the genus is the maximum Alexander grading in the support of~$\CFK(K)$. We say that $K$ is \emph{$\gamma_0$-sharp} if the genus of~$K$ is the maximum Alexander grading (i.e., the maximum vertical coordinate) in the support of~$\gamma_0$.

\begin{remark}\label{rmk:tightfiberedgammashparp}
Observe that the maximum Alexander grading in the support of $\gamma_0$ is at least $|\tau(K)|$, where $\tau$ is the Ozsv\'ath-Szab\'o concordance invariant from \cite{Oz-Sz:2003}. Hence if $g(K)=\tau(K)$, then $K$ is $\gamma_0$-sharp.
\end{remark}

We illustrate these ideas with the following examples.

\begin{example}
Consider the knot $K=T_{2,3} \# T_{2,3}$. Its knot Floer complex is depicted below:

\[
\begin{tikzpicture}[scale=1]
	\draw[step=1, black!30!white, very thin] (-1.5, -0.5) grid (4.5, 3.5);

	\filldraw (0.3, 2.7) circle (2pt) node[label=above:{\lab{U^2 a}}] (a) {};
	\filldraw (1.3, 2.7) circle (2pt) node[label=above:{\lab{U b}}] (b) {};
	\filldraw (1.3, 1.3) circle (2pt) node[label=left:{\lab{U V c}}] (c) {};
	\filldraw (2.7, 1.3) circle (2pt) node[label=right:{\lab{V d}}] (d) {};
	\filldraw (2.7, 0.3) circle (2pt) node[label=right:{\lab{V^2 e}}] (e) {};

	\draw[->] (b) to (a);
	\draw[->] (b) to (c);
	\draw[->] (d) to (c);
	\draw[->] (d) to (e);

	\filldraw (2.7, 2.7) circle (2pt) node[label=above:{\lab{f}}] (f) {};
	\filldraw (1.7, 2.7) circle (2pt) node[label=above:{\lab{U g}}] (g) {};
	\filldraw (1.7, 1.7) circle (2pt) node[label=left:{\lab{U V i}}] (i) {};
	\filldraw (2.7, 1.7) circle (2pt) node[label=right:{\lab{UV h}}] (h) {};
	
	\draw[->] (f) to (g);
	\draw[->] (f) to (h);
	\draw[->] (g) to (i);
	\draw[->] (h) to (i);

\end{tikzpicture}
\]
Below left is the immersed curve for $K=T_{2,3} \# T_{2,3}$ and below right is the immersed curve $\gamma_0$:
\[ 
   \labellist
    \small\hair 2pt
    \pinlabel $\gamma_0$ at 87 120
    \pinlabel $\gamma_1$ at 85 85
    \pinlabel $\gamma_0$ at 230 120
    \endlabellist
\includegraphics{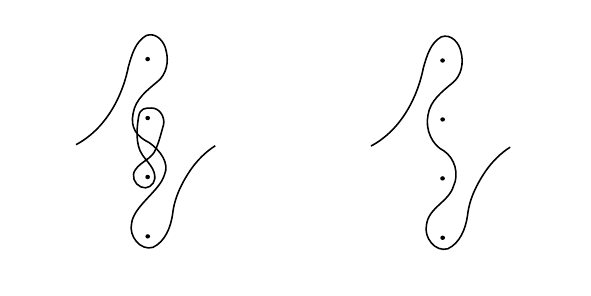} 
\]
This knot is $\gamma_0$-sharp, since $g(K)=2$ and the maximum Alexander grading in the support of $\gamma_0$ is also $2$.
\end{example}

\begin{example}
Consider the knot $K=T_{2,3;2,-1}$, the $(2,-1)$-cable of the right-handed trefoil. Following \cite{Hanselman-Watson:2023}, we have the immersed curve for $K$ looks like:
\[
\includegraphics{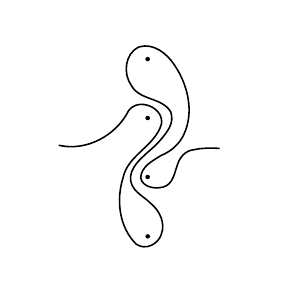} 
\]
which is parametrized by the sequence
\[ (1, -2, -1, 1, -1, 1, 2, -1). \]
The sign convention on the sequence can be interpreted as follows: moving along $\gamma_0$ from left to right, a positive entry indicates curling around the pegs in a clockwise direction, while a negative entry indicates curling around the pegs in a counterclockwise direction.
The knot Floer complex of $K$ is depicted below:
\[
\begin{tikzpicture}[scale=1]
	\draw[step=1, black!30!white, very thin] (-2.5, -0.5) grid (3.5, 3.5);

	\filldraw (0.5, 2.65) circle (2pt) node[label=above:{\lab{U a}}] (a) {};
	\filldraw (1.65, 2.65) circle (2pt) node[label=above:{\lab{b}}] (b) {};
	\filldraw (1.65, 0.5) circle (2pt) node[label=below:{\lab{V^2 c}}] (c) {};
	\filldraw (0.5, 0.5) circle (2pt) node[label=left:{\lab{UV^2 d}}] (d) {};
	\filldraw (0.5, 1.5) circle (2pt) node[label= below left:{\lab{UV e}}] (e) {};
	\filldraw (-0.5, 1.5) circle (2pt) node[label=left:{\lab{U^2V f}}] (f) {};
	\filldraw (-0.5, 2.35) circle (2pt) node[label=left:{\lab{U^2 g}}] (g) {};
	\filldraw (1.35, 2.35) circle (2pt) node[label=right:{\lab{h}}] (h) {};
	\filldraw (1.35, 1.5) circle (2pt) node[label=right:{\lab{V i}}] (i) {};

	\draw[->] (b) to (a);
	\draw[->] (b) to (c);
	\draw[<-] (d) to (c);
	\draw[<-] (d) to (e);
	\draw[->] (f) to (e);
	\draw[->] (e) to (f);
	\draw[->] (g) to (f);	
	\draw[->] (h) to (g);
	\draw[->] (h) to (i);	
	
	\draw[->] (a) to (f);	
	\draw[->] (b) to [bend right] (e);	
	\draw[->] (h) to (e);	
	\draw[->] (i) to (d);	

\end{tikzpicture}
\]
Here, the sign convention on the sequence indicates whether we are moving against the direction of an arrow (for a positive entry) or with the direction of an arrow (for a negative entry), starting from the generator with no incoming or outgoing vertical arrows.
The diagonal arrows are necessary to ensure that $\d^2=0$. This knot is $\gamma_0$-sharp, since both its genus and the maximum Alexander grading support of $\gamma_0$ is 2.
\end{example}

In contrast to the above examples, any nontrivial slice knot $K$ is not $\gamma_0$-sharp, since $\gamma_0$ for a slice knot is the immersed curve associated to the unknot, that is, a horizontal line, and thus the maxiumum Alexander grading in the support of $\gamma_0$ is zero, which does not equal the genus of $K$ if $K$ is nontrivial. 

Our primary computation using knot Floer homology will involve \emph{L-space knots}, that is, knots in $S^3$ that admit a positive surgery to an L-space. If $K$ is an L-space knot, then its connected complex (over either $\F[U,V]$ or $\F[U,V]/(UV)$) is exactly  its knot Floer complex $\CFK(K)$. Hence, for an L-space knot $K$, the knot Floer complex is equal to its standard complex representative, which can be completely described by a sequence of integers. Moreover, this sequence can be read off from the Alexander polynomial of $K$.

We illustrate this with an example.

\begin{example}
Consider $K = T_{4,5}$, whose Alexander polynomial is 
\[ \Delta_K(t) = t^6-t^5+t^2-1+t^{-2}-t^{-5}+t^6. \]
Its standard complex is generated by $x_0, x_1, x_2, x_3, x_4, x_5, x_6$, with nonzero differentials
\[ \d x_1 = U x_0 + V^3 x_2, \qquad \qquad \d x_3 = U^2 x_2 + V^2 x_4, \qquad \qquad \d x_5 = U^3 x_4 + V x_6, \]
which is parametrized by the sequence
\[ (1, -3, 2, -2, 3, -1). \]
The immersed curve invariant of $K$ is exactly $\gamma_0$:
\[
\includegraphics{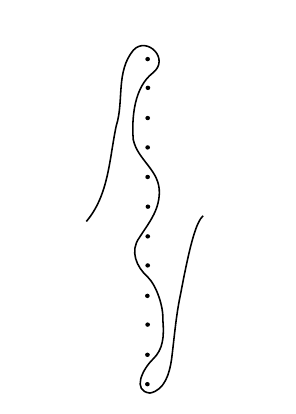}
\]
Note that the sequence $$(a_1, b_1, a_2, b_2, a_3, b_3)=(1, -3, 2, -2, 3, -1)$$ can be observed in the immersed curve as follows: starting from the top, we first see an arc to the right of the pegs (a \emph{right arc}) of length $a_1=1$, followed by an arc to the left of the pegs (a \emph{left arc}) of length $|b_1| = 3$, etc.
\end{example}

Recall that the invariant $\CFK(K)$ can be endowed with an endomorphism $\iota_K$, as in \cite{HendricksManolescu}. In the case where $K$ is an L-space knot, with Alexander polynomial
\[ \Delta_K(t) = t^{a_0} - t^{a_1} + \dots + t^{a_n} \]
for some even integer $n$, we have that $\CFK(K)$ is the standard complex generated by $x_0, \dots, x_n$, with nonzero differentials
\[ \d x_i = U^{a_{i-1}-a_i} x_{i-1} + V^{a_{i+1}-a_i} x_{i+1} \]
for $i$ odd. The \emph{basic involution} on this complex interchanges $x_i$ with $x_{n-i}$, and by \cite[Section 7]{HendricksManolescu}, the basic involution is exactly $\iota_K$.






We conclude this section by showing that satelliting with a 1-bridge braid pattern preserves $\gamma_0$-sharpness:

\begin{proposition}\label{prop:preservedcabling}
Let $P$ be a 1-bridge braid pattern in the solid torus.  
If $K$ is $\gamma_0$-sharp, then $P(K)$ is also $\gamma_0$-sharp.  
In particular, any cable of $K$ is $\gamma_0$-sharp.
\end{proposition}

\begin{proof}
By \cite[Theorem 1.3]{ChenHanselman}, we have that for $P$ a 1-bridge braid pattern in the solid torus, $\gamma(P(K))$ is homotopic to $f_P(\gamma(P(K)))$, where $f_P$ is the diffeomorphism from \cite[Section 6.4]{ChenHanselman}. If the Seifert genus of $K$ is determined by $\gamma_0$, then the compact components of $\gamma(K)$ all lie within the vertical span of $\gamma_0(K)$, and this property is preserved by $f_P$. It follows that the Seifert genus of $P(K)$ is also determined by $\gamma_0$, as desired.
\end{proof}

\section{Proof of the main theorem}\label{sec:mainthm}

%

We now prove the main theorem, whose statement we recall below.  
The argument follows closely that of Baker~\cite[Theorem~3]{Baker:2016}.

\setcounter{theorem}{0}
\renewcommand{\thetheorem}{1.1}
\begin{theorem}\label{thm:mainbody}
A connected sum of $\gamma_0$-sharp fibered knots  is ribbon if and only if it is of the form $K \mathbin{\#} -K$.
\end{theorem}
\renewcommand{\thetheorem}{\thesection.\arabic{theorem}}

\begin{proof}
Assume that $J$ is a $\gamma_0$-sharp fibered knot.  
We will show that $J$ is minimal with respect to homotopy ribbon concordance among all fibered knots in $S^3$, thereby generalizing~\cite[Lemma~2]{Baker:2016}.  
Suppose that $J$ is homotopy ribbon concordant to a fibered knot $J'$ in $S^3$.  
Then, by~\cite[Lemma~3.4]{Gordon:1981}, we have either $J = J'$ or $g(J) > g(J')$.  
Although the result is stated for ribbon concordance, the same argument applies to homotopy ribbon concordance.  
By the assumption that $J$ is $\gamma_0$-sharp and by Remark~\ref{rmk:homologyconcordance}, we have $g(J) \leq g(J')$, which implies that $J = J'$.

We now apply~\cite[Theorem~5.5]{Miyazaki:1994}, which states that if a connected sum of knots that are minimal with respect to homotopy ribbon concordance among all fibered knots in homology spheres is ribbon, then it must be of the form $K \mathbin{\#} -K$. On the other hand, the remark following~\cite[Lemma~1.3]{Miyazaki:1994} and the proof of~\cite[Theorem~3]{Baker:2016} state that the condition ``among all fibered knots in homology spheres'' can be relaxed to ``among all fibered knots in~$S^3$'' using the Geometrization Conjecture~\cite{Perelman:2002-1, Perelman:2003-1, Perelman:2003-2}.  
This completes the proof.
 \end{proof}

\begin{remark}
If we consider ribbon concordance instead of homotopy ribbon concordance,  
then each $\gamma_0$-sharp fibered knot $J$  is minimal not only among all fibered knots in $S^3$,  
but also among all knots in $S^3$.  
This follows from an observation of Miyazaki~\cite{Miyazaki:2018}, which builds on results of Silver~\cite{Silver:1992} and Kochloukova~\cite{Kochloukova:2006},  
showing that if there is a ribbon concordance from $J$ to $J'$ and $J$ is fibered, then $J'$ must also be fibered.
\end{remark}

Next, we prove Corollary~\ref{cor:main}. By iterated cables, we mean knots obtained by performing at least one nontrivial cabling operation. The statement is as follows:

\setcounter{corollary}{1}
\renewcommand{\thecorollary}{1.2}
\begin{corollary}\label{cor:mainbody}
Either distinct iterated cables of tight fibered knots are linearly independent in the smooth knot concordance group, or the slice-ribbon conjecture is false.
\end{corollary}

\begin{proof}
Assume the slice--ribbon conjecture is true.  
Let $K_1, K_2, \dots, K_n$ be distinct iterated cables of tight fibered knots, and suppose that  
\[
K := c_1K_1 \# c_2K_2 \# \cdots \# c_nK_n
\]
is slice, and hence ribbon, where $c_iK_i$ denotes the $c_i$-fold connected sum of $K_i$ if $c_i > 0$, the $(-c_i)$-fold connected sum of $-K_i$ if $c_i < 0$, and the unknot if $c_i = 0$.  
Without loss of generality, we may assume that $c_1 \geq 0$. It suffices to show that $c_1 = 0$.

Recall that a tight fibered knot $J$ satisfies $\tau(J) = g(J)$ by~\cite[Theorem~4]{Livingston:2004}.  
Moreover, by~\cite[Corollary~4 and Theorem~2]{Hom:2014}, we have $\varepsilon(J) = 1$, and any iterated cable of $J$ also satisfies $\varepsilon = 1$.  
In particular, for each $i$, we have $\varepsilon(K_i) = 1$ and $\varepsilon(-K_i) = -1$~\cite[Proposition~3.6]{Hom:2014}. If $c_1 > 0$, then since $K$ is ribbon and the knots $K_1, \dots, K_n$ are assumed to be distinct, Theorem~\ref{thm:main} implies that $K_1 = -K_i$ for some $i$.  
However, this leads to a contradiction, since $\varepsilon(K_1) = 1$ while $\varepsilon(-K_i) = -1$.  
Therefore, we must have $c_1 = 0$.
\end{proof}

Finally, we prove Corollary~\ref{cor:examples}, restated below for convenience:

\setcounter{corollary}{1}
\renewcommand{\thecorollary}{1.3}
\begin{corollary}\label{cor:examplesbody}
Let $K$ and $J$ be distinct iterated cables of tight fibered knots. Then, for each pair of distinct integers $q_1$ and $q_2$ that are relatively prime to $p$, the knot
\[
P(K, J, p, q_1, q_2) := K_{p,q_1} \mathbin{\#} -K_{p,q_2} \mathbin{\#} J_{p,q_2} \mathbin{\#} -J_{p,q_1}
\]
is algebraically slice but not ribbon.
\end{corollary}
\renewcommand{\thecorollary}{\thesection.\arabic{corollary}}

\begin{proof}
By Corollary~\ref{cor:main} and the fact that $P(K, J, p, q_1, q_2)$ is known to be algebraically slice by~\cite{Livingston-Melvin:1983}, it suffices to show that $K_{p,q_1}$, $K_{p,q_2}$, $J_{p,q_1}$, and $J_{p,q_2}$ are distinct. This will follow in particular from showing that $K_{p,q_1}$ is not isotopic to any of $K_{p,q_2}$, $J_{p,q_1}$, or $J_{p,q_2}$ and then applying symmetry.

The first case is immediate, since we are assuming $q_1 \neq q_2$, which implies that $K_{p,q_1}$ and $K_{p,q_2}$ have distinct Alexander polynomials.   For the remaining cases, recall that if $K_{p,q}$ is isotopic to $J_{p',q'}$ for some $p, p' > 1$, then $K$ is isotopic to $J$. This follows because any homeomorphism between the complements of the cabled knots induces a homeomorphism between the complements of $K$ and $J$, by the uniqueness of the JSJ decomposition~\cite{Jaco-Shalen:1979, Johannson:1979}, and knot complements determine knot types~\cite{Gordon-Luecke:1989}.   This completes the proof.
\end{proof}

%
%
%

\section{Proof of Proposition \ref{prop:examples}}\label{sec:computations}

The goal of this section is to prove Proposition~\ref{prop:examples}, namely, to determine when the concordance invariants derived from the knot Floer complex vanish for
\[
P(K, U, 2, q_1, q_2) := K_{2,q_1} \mathbin{\#} -K_{2,q_2} \mathbin{\#} T_{2,q_2} \mathbin{\#} -T_{2,q_1}.
\]
In Section~\ref{sec:background}, we recalled that the connected complex over $\F[U,V]/(UV)$ of $K$ can be described by a sequence of nonzero integers, and that for an L-space knot, the connected complex coincides with $\CFK(K)$. The following proposition describes the effect of $(2,q)$-cabling on an L-space knot $K$.

%

%



\begin{proposition}\label{prop:cableL}
Let $K$ be an L-space knot whose knot Floer complex is described by the integer sequence $$(a_1, b_1, \dots, a_n, b_n).$$
\begin{enumerate}
	\item\label{it:cableL1} If $q > 4g(K)$, then the knot Floer complex of the cable $K_{2,q}$ is described by the sequence
	\[ (a'_1, b'_1, \dots, a'_n, b'_n, 1, -1, 1, -1, \dots, 1, -1, -b'_n, -a'_n, \dots, -b'_1, -a'_1), \]
	where 
	\begin{enumerate}
		\item $a'_i$ is a sequence of length $2a_i-1$ of alternating $\pm 1$, starting with $1$,
		\item $b'_i = 2b_i-1$, and
		\item the subsequence $1, -1, \dotsm 1, -1$ in the middle has length $q-4g(K)-1$.
	\end{enumerate}
	\item\label{it:cableL2} If $q < 4g(K)$, then the connected complex of the cable $K_{2,q}$ is described by the sequence
	\[ (a'_1, b'_1, \dots, a'_n, b'_n, -1, 1, -1, 1, \dots, -1, 1, -b'_n, -a'_n, \dots, -b'_1, -a'_1), \]
	where 
	\begin{enumerate}
		\item $a'_i$ is a sequence of length $2a_i-1$ of alternating $\pm 1$, starting with $1$,
		\item $b'_i = 2b_i-1$ if $i \neq n$ and $b'_n = 2b_n$, and
		\item the subsequence $-1, 1, \dotsm -1, 1$ in the middle has length $4g(K)-1-q$.
	\end{enumerate}
\end{enumerate}
\end{proposition}

\begin{example}
Consider the L-space knot $T_{4,5}$. The sequence description for the knot Floer complex of $T_{4,5}$ is $$(1, -3, 2, -2, 3, -1)$$ and its genus is $6$. Then the $(2, 27)$-cable of $T_{4,5}$ is described by the sequence
\[ (\underbrace{1}_{a'_1}, \underbrace{-7}_{b'_1}, \underbrace{1, -1, 1}_{a'_2}, \underbrace{-5}_{b'_2}, \underbrace{1, -1, 1, -1, 1}_{a'_3}, \underbrace{-3}_{b'_3}, \underbrace{1, -1}_{\textup{middle}}, \underbrace{3}_{-b'_3}, \underbrace{-1, 1, -1, 1, -1}_{-a'_3}, \underbrace{5}_{-b'_2}, \underbrace{-1, 1, -1}_{-a'_2}, \underbrace{7}_{-b'_1}, \underbrace{-1}_{-a'_1}), \]
where the middle sequence has length $27-4g(T_{4,5})-1=2$.
\end{example}

\begin{proof}
The proof relies on Hanselman-Watson's cabling formula \cite{Hanselman-Watson:2023} in terms of immersed curves.

For a $(2,q)$-cable, their formula is as follows. Take a copy of the immersed curve $\gamma$ for $K$ and stretch it by a factor of two; call the resulting curve $\gamma'_1$. Now to the right of $\gamma'_1$, take a second copy of the immersed curve $\gamma$, stretch it by a factor of two, and shift down by $q$ units; call the resulting curve $\gamma'_2$. Connect the right endpoint of $\gamma'_1$ to the left endpoint of $\gamma'_2$ by an arc $\gamma'_{1.5}$. Now compress the picture horizontally so that the pegs for $\gamma_1$ and $\gamma_2$ lie in a single vertical line. Let $\gamma_i$, $i=1, 1.5, 2$ denote the image of $\gamma'_i$ after horizontal compression.

We first consider the case $q > 4g(K)$. Consider $\gamma_1$. The $a_i$ refer to the right arcs of $\gamma$ and we see that horizontal compression turns a right arc of length $a_i$ into a sequence of right and left arcs of length one, where there are $a_i$ right arcs and $a_i-1$ left arcs. Similarly, the $b_i$ refer to the left arcs of $\gamma$, and now the horizontal compression turns a left arc of length $b_i$ into a left arc of length $2b_1+1$. 

Next, we consider $\gamma_{1.5}$. Since $\gamma'_{1.5}$ lies in a vertical strip between the pegs of $\gamma'_1$ and $\gamma'_2$, horizontal compression results in a sequence of right and left arcs of length one. The total number of such right and left arcs is $q-4g(K)-1$.

Lastly, we consider $\gamma_2$, which follows as for $\gamma_1$, with the roles of $a_i$ and $b_i$ reversed.

The case $q < 4g(K)$ is similar.
\end{proof}

We now give a connected sum formula for certain standard complexes. 

\begin{proposition}\label{prop:K+T}
Let $K$ be an L-space knot whose knot Floer complex is described by the integer sequence
\[ (a_1, b_1, \dots, a_n, b_n, -b_n, -a_n, \dots, -b_1, -a_1)\]
where $a_i=1$ for $i=1, \dots, n$. Let $q>2$. Then $K\# T_{2,q}$ is locally equivalent to 
\[ (a_1, b_1, \dots, a_n, b_n, \underbrace{1, -1, 1, -1, \dots, 1, -1}_{q-1\textup{ entries}}, -b_n, -a_n, \dots, -b_1, -a_1)\]
and $K\# T_{2,-q}$ is locally equivalent to 
\[ (a_1, b_1, \dots, a_n, b_n, \underbrace{-1, 1, -1, 1, \dots, -1, 1}_{q-1\textup{ entries}}, -b_n, -a_n, \dots, -b_1, -a_1).\]
Furthermore, $\iota_K$ on $K\# T_{2,q}$ is given by the basic involution.
\end{proposition}

We first consider $K \# T_{2,q}$, where $q = 2k + 1>2$.
Let  
\[
\{x_0, x_1, \dots, x_{2n-1}, x_{2n}, x'_{2n-1}, \dots, x'_0\}
\]  
be the standard basis for $K$, and let  
\[
\{y_0, y_1, \dots, y_{k-1}, y_k, y'_{k-1}, \dots, y'_0\}
\]  
be the standard basis for $T_{2,q}$.  Then for $i=0, 1, \dots 2n-1$,
\[ x_i \xleftrightarrow{\iota_K} x'_i \]
and for $j=0, 1, \dots, k-1$, 
\[ y_j \xleftrightarrow{\iota_K} y'_j \]
and
\[ \iota_K(x_{2n})= x_{2n}, \quad \iota_K(y_k)=y_k. \]

Let
\begin{gather}
\begin{aligned}\label{eqn:X}
	X = \{&x_0y_0, \; x_1y_0, \dots, \; x_{2n-1}y_0, \; x_{2n}y_0, \\
		& x_{2n}y_1, \; \dots, \; x_{2n}y_{k-1}, \; x_{2n}y_k, \; x_{2n}y'_{k-1}, \dots, \; x_{2n}y'_1, \\
		& x'_{2n}y'_0, \; x'_{2n-1}y'_0,  \dots, x'_0y'_0\}.
\end{aligned}
\end{gather}
For $i=1, 3, \dots, 2n-1$ and $j=1, 3, \dots, k'$, where throughout $k'$ is the greatest odd number less than or equal to $k$, let
\[ Y_{i, j} = \left\{ x_i y_j, x_{i-1}y_{j}+x_i y_{j-1}, x_{i} y_{j+1} + V^{b_1-1}x_{i+1}y_{j}, x_{i-1}y_{j+1}+V^{b_1-1}x_{i+1} y_{j-1}  \right\}, \]
which we can visualize as
\[
Y_{i,j} = \begin{tikzcd}
	 x_{i-1}y_{j}+x_i y_{j-1} \ar[d, "V"] & x_i y_j \ar[l, "U"] \ar[d, "V"]  \\
	 x_{i-1}y_{j+1}+V^{b_i-1}x_{i+1} y_{j-1} & x_i y_{j+1} + V^{b_i-1}x_{i+1}y_{j}. \ar[l, "U"] 
\end{tikzcd}
\]
(Note that throughout, we mildly abuse notation, and if $j=k'=k$, then by $y_{j+1}$, we actually mean $y'_{j-1}$.)
Similarly, for $i=1, 3, \dots, 2n-1$ and $j=1, 3, \dots, k'-2$, we have
\[ Z_{i, j} = \left\{ x_i y'_j, x_{i-1}y'_{j}+x_i y'_{j+1}, x_{i} y'_{j-1} + V^{b_1-1}x_{i+1}y'_{j}, x_{i-1}y'_{j-1}+V^{b_1-1}x_{i+1} y'_{j+1}  \right\}, \]
which we can visualize as
\[
Z_{i,j} = \begin{tikzcd}
	 x_{i-1}y'_{j}+x_i y'_{j+1} \ar[d, "V"] & x_i y'_j \ar[l, "U"] \ar[d, "V"]  \\
	 x_{i-1}y'_{j-1}+V^{b_i-1}x_{i+1} y'_{j+1} & x_i y'_{j-1} + V^{b_i-1}x_{i+1}y'_{j}. \ar[l, "U"] 
\end{tikzcd}
\]
Symmetrically, for $i=1, 3, \dots, 2n-1$ and $j=1, 3, \dots, k'$, we have nearly the same thing, but with a slightly modification in order for $\iota_K$ to act nicely on this choice of basis:
\[ Y'_{i,j} =  \left\{ x'_i y'_j+ b_i U^{b_i-1}x'_{i+1} y'_{j-1}, x'_{i-1}y'_{j}+x'_i y'_{j-1}, x'_{i} y'_{j+1} + U^{b_i-1}x'_{i+1}y'_{j}, x'_{i-1}y'_{j+1}+U^{b_i-1}x'_{i+1} y'_{j-1} \right\}, \]
which we can visualize as
\[
Y'_{i,j} = \begin{tikzcd}
	 x'_i y'_{j+1} + U^{b_i-1}x'_{i+1}y'_{j} \ar[d, "V"] & x'_i y'_j + b_i U^{b_i-1}x'_{i+1} y'_{j-1}\ar[l, "U"] \ar[d, "V"]  \\
	 x'_{i-1}y'_{j+1}+U^{b_i-1}x'_{i+1} y'_{j-1} & x'_{i-1}y'_{j}+x'_i y'_{j-1} .\ar[l, "U"] 
\end{tikzcd}
\]
(Again, we mildly abuse notation, and if $j=k'=k$, then by $y'_j$, we actually mean $y_j$, and by $y'_{j+1}$, we actually mean $y_{j-1}$.)
For $i=1, 3, \dots, 2n-1$ and $j=1, 3, \dots, k'-2$, we have
\[ Z'_{i, j} = \left\{ x'_i y_j +  b_i U^{b_i-1}x'_{i+1} y_{j-1}, x'_{i-1}y_{j}+x'_i y_{j+1}, x'_{i} y_{j-1} + U^{b_i-1}x'_{i+1}y_{j}, x'_{i-1}y_{j-1}+U^{b_i-1}x'_{i+1} y_{j+1}  \right\}, \]
which we can visualize as
\[
Z'_{i,j} = \begin{tikzcd}
	   x'_i y_{j-1} + U^{b_i-1}x'_{i+1}y_{j} \ar[d, "V"] & x'_i y_j +  b_i U^{b_i-1}x'_{i+1} y_{j-1} \ar[l, "U"] \ar[d, "V"]  \\
	 x'_{i-1}y_{j-1}+U^{b_i-1}x'_{i+1} y_{j+1} & x'_{i-1}y_{j}+x'_i y_{j+1}. \ar[l, "U"] 
\end{tikzcd}
\]

It is straightforward to observe that $X, Y_{i,j}, Z_{i,j}, Y'_{i,j},$ and $Z'_{i,j}$ form a basis for $\CFK(K \# T_{2,q})$.

\begin{lemma}\label{lem:X}
The subcomplex $X$ is described by the sequence 
\[ (a_1, b_1, \dots, a_n, b_n, \underbrace{1, -1, 1, -1, \dots, 1, -1}_{q-1\textup{ entries}}, -b_n, -a_n, \dots, -b_1, -a_1),\]
and the map $\iota_K$ on $X$ is the basic involution.
\end{lemma}

\begin{proof}
Computing the differential on $X$, one sees that $X$ has the sequence description above.

We now consider $\iota_K$ on $X$. Note that since $X$ consists of elements of the form $x_i y_j$, $x_i y'_j$, and $x'_i y'_j$ where at least one of the subscripts is even, the $(\Phi \otimes \Psi) \circ (\iota_1 \otimes \iota_2)$ term always vanishes. Hence we have that $\iota_K$ sends the $\ell^\textup{th}$ entry of \eqref{eqn:X} to the $\ell^{\textup{th}}$ from the end, which is exactly the basic involution.
\end{proof}

\begin{lemma}\label{lem:YZ}
The map $\iota_K$ takes $Y_{i,j}$ to $Y'_{i,j}$ and $Z_{i,j}$ to $Z'_{i,j}$.
\end{lemma}

\begin{proof}
We first consider $Y_{i,j}$ and $Y'_{i,j}$. Recall that $i$ and $j$ are odd. On $Y_{i,j}$, we have
\begin{align*}
	\iota_K(x_iy_j) &= x'_i y'_j + b_i U^{b_i-1}x'_{i+1} y'_{j-1} \\
	\iota_K( x_{i-1}y_{j}+x_i y_{j-1}) &= x'_{i-1}y'_{j}+x'_i y'_{j-1} \\	
	\iota_K( x_i y_{j+1} + V^{b_i-1}x_{i+1}y_{j}) &= x'_i y'_{j+1} + U^{b_i-1}x'_{i+1}y'_{j} \\	
	\iota_K( x_{i-1}y_{j+1}+V^{b_i-1}x_{i+1} y_{j-1}) &= x'_{i-1}y'_{j+1}+U^{b_i-1}x'_{i+1} y'_{j-1},
\end{align*}
and on $Y'_{i,j}$, we have
\begin{align*}
	\iota_K(x'_iy'_j+ b_i U^{b_i-1}x'_{i+1} y'_{j-1} ) &= x_i y_j + 2b_i V^{b_i-1}x_{i+1} y_{j-1} \\
	\iota_K( x'_{i-1}y'_{j}+x'_i y'_{j-1}) &= x_{i-1}y_{j}+x_i y_{j-1} \\	
	\iota_K( x'_i y'_{j+1} + U^{b_i-1}x'_{i+1}y'_{j}) &= x_i y_{j+1} + U^{b_i-1}x_{i+1}y_{j} \\	
	\iota_K( x'_{i-1}y'_{j+1}+U^{b_i-1}x'_{i+1} y'_{j-1}) &= x_{i-1}y_{j+1}+U^{b_i-1}x_{i+1} y_{j-1}.
\end{align*}

Similarly, we consider $\iota_K$ on $Z_{i,j}$ and $Z'_{i,j}$. Again, recall that $i$ and $j$ are odd. On $Z_{i,j}$, we have
\begin{align*}
	\iota_K(x_i y'_j ) &= x'_i y_j +  b_i U^{b_i-1}x'_{i+1} y_{j-1} \\
	\iota_K(  x_{i-1}y'_{j}+x_i y'_{j+1}) &=   x'_{i-1}y_{j}+x'_i y_{j+1} \\	
	\iota_K( x_i y'_{j-1} + V^{b_i-1}x_{i+1}y'_{j}) &= x'_i y_{j-1} + U^{b_i-1}x'_{i+1}y_{j} \\	
	\iota_K(  x_{i-1}y'_{j-1}+V^{b_i-1}x_{i+1} y'_{j+1} ) &= x'_{i-1}y_{j-1}+U^{b_i-1}x'_{i+1} y_{j+1},
\end{align*}
and on $Z'_{i,j}$, we have
\begin{align*}
	\iota_K(x'_i y_j +  b_i U^{b_i-1}x'_{i+1} y_{j-1}) &=  x_i y'_j\\
	\iota_K( x'_{i-1}y_{j}+x'_i y_{j+1}) &=  x_{i-1}y'_{j}+x_i y'_{j+1}  \\	
	\iota_K( x'_i y_{j-1} + U^{b_i-1}x'_{i+1}y_{j}) &= x_i y'_{j-1} + V^{b_i-1}x_{i+1}y'_{j} \\	
	\iota_K(  x'_{i-1}y_{j-1}+U^{b_i-1}x'_{i+1} y_{j+1}) &=  x_{i-1}y'_{j-1}+V^{b_i-1}x_{i+1} y'_{j+1},
\end{align*}
which completes the proof.
\end{proof}

Proposition \ref{prop:K+T} now follows readily:
\begin{proof}[Proof of Proposition \ref{prop:K+T}]
It follows from Lemmas \ref{lem:X} and \ref{lem:YZ} that we have an $\iota_K$-equivariant splitting of $\CFK(K \# T_{2,q})$ with $X$ as a direct summand. The result now follows from the description of $X$. 

The statement for $K \# T_{2,-q}$ follows from an analogous choice of basis; however, since $K \# T_{2,-q}$ is not locally equivalent to an L-space knot, we do not have a notion of a basic involution.
\end{proof}

Finally, we are ready to prove Proposition~\ref{prop:examples}:

\setcounter{proposition}{1}
\renewcommand{\theproposition}{1.4}
\begin{proposition}\label{prop:examples}

Let $K$ be a nontrivial L-space knot, and let $U$ denote the unknot. Then, for each pair of distinct odd integers $q_1$ and $q_2$, the knot
\[
P(K, U, 2, q_1, q_2) := K_{2,q_1} \mathbin{\#} -K_{2,q_2} \mathbin{\#} T_{2,q_2} \mathbin{\#} -T_{2,q_1}
\]
is algebraically slice but not ribbon. 
Moreover, its knot Floer complex  is locally equivalent to that of the unknot if and only if the pair $(q_1, q_2)$ satisfies one of the following:
\begin{itemize}
    \item $q_1, q_2 > 4g(K)$,
    \item $0 < q_1, q_2 < 4g(K)$, or
    \item $q_1, q_2 < 0$,
\end{itemize}
where $g(K)$ is the Seifert genus of $K$. In particular, in all of the above cases, all concordance invariants derived from the knot Floer complex vanish.

Furthermore, when $q_1, q_2 > 4g(K)$, the obstruction coming from the involutive knot Floer complex vanishes. When $q_1, q_2 < 0$, the obstruction coming from the Ozsv\'{a}th--Szab\'{o} $d$-invariant, together with the Hendricks--Manolescu $\overline{d}$- and $\underline{d}$-invariants of the double branched covers, vanish.

\end{proposition}
\renewcommand{\theproposition}{\thesection.\arabic{proposition}}


\begin{proof}
We will first show that the knot Floer complexes of $K_{2, q_1} \# T_{2, q_2}$ and $K_{2, q_2} \# T_{2, q_1}$ are locally equivalent if and only if the pair $(q_1, q_2)$ satisfies one of the conditions listed above.

For the ``if'' direction, let $K$ be an L-space knot described by the integer sequence $$(a_1, b_1, \dots, a_n, b_n).$$
 Suppose that $q_1, q_2 > 4g(K)$. Observe that if $q_1> 4g(K)$, then $K_{2,q_1}$ is an L-space knot described by the sequence 
\[ (a'_1, b'_1, \dots, a'_n, b'_n, \underbrace{1, -1, 1, -1, \dots, 1, -1}_{q_1-4g(K)-1\textup{ entries}}, -b'_n, -a'_n, \dots, -b'_1, -a'_1) \]
where $a'_i, b'_i$ are as in Proposition \ref{prop:cableL}\eqref{it:cableL1}. Since this sequence satisfies the hypothesis of Proposition~\ref{prop:K+T}, we have that $K_{2,q_1} \# T_{2, q_2}$ is $\iota_K$-locally equivalent to 
\[ (a'_1, b'_1, \dots, a'_n, b'_n, \underbrace{1, -1, 1, -1, \dots, 1, -1}_{q_1+q_2-4g(K)-2\textup{ entries}}, -b'_n, -a'_n, \dots, -b'_1, -a'_1). \]
Reversing the roles of $q_1$ and $q_2$ gives the desired result.

The non-involutive proof for $q_1, q_2 < 4g(K)$ is similar. Indeed, let $K$ be an L-space knot described by the integer sequence $$(a_1, b_1, \dots, a_n, b_n).$$ The connected complex of  $K_{2, q_1}$ is given by
\[ (a'_1, b'_1, \dots, a'_n, b'_n, \underbrace{-1, 1, -1, 1, \dots, -1, 1}_{4g(K)-q_1-1\textup{ entries}}, -b'_n, -a'_n, \dots, -b'_1, -a'_1), \]
where $a'_i, b'_i$ are as in Proposition \ref{prop:cableL}\eqref{it:cableL2}. 


Suppose that $0 < q_2 < 4g(K)$. The complex for $K_{2, q_1}$ is given by
\begin{equation}\label{eq:K2q1}
(a'_1, b'_1, \dots, a'_n, b'_n, \underbrace{-1, 1, -1, 1, \dots, -1, 1}_{4g(K)-q_1-1 \textup{ entries}}, -b'_n, -a'_n, \dots, -b'_1, -a'_1)
\end{equation}
while the complex for $T_{2, q_2}$ is given by
\begin{equation}\label{eq:T2q2}
 (\underbrace{1, -1, 1, -1, \dots, 1, -1}_{q_2-1\textup{ entries}}).
\end{equation}
If $q_2 \geq 4g(K) - q_1$, then the number of consecutive $\pm 1$'s in the middle of~\eqref{eq:K2q1} is at most equal to that in~\eqref{eq:T2q2}, effectively annihilating $4g(J)-q_1-1$ of the entries, yielding a middle sequence of length $q_2-1-(4g(K)-q_1-1)$ and so the connected complex of  $K_{2, q_1} \# T_{2, q_2}$ is
\[ (a'_1, b'_1, \dots, a'_n, b'_n, \underbrace{1, -1, 1, -1, \dots, 1, -1}_{q_1+q_2 -4g(K)\textup{ entries}}, -b'_n, -a'_n, \dots, -b'_1, -a'_1). \]
On the other hand, if $0<q_2 < 4g(K)-q_1$, then the middle sequence of $\pm 1$ in \eqref{eq:K2q1} is longer than the sequence in \eqref{eq:T2q2}, yielding the connected complex of  $K_{2, q_1} \# T_{2, q_2}$ as
\[ (a'_1, b'_1, \dots, a'_n, b'_n, \underbrace{-1, 1, -1, 1, \dots, -1, 1}_{4g(K)-q_1-q_2\textup{ entries}}, -b'_n, -a'_n, \dots, -b'_1, -a'_1).\]
Now suppose that $q_2<0$. Then the connected complex of  $K_{2, q_1} \# T_{2, q_2}$ is given by
\[ (a'_1, b'_1, \dots, a'_n, b'_n, \underbrace{-1, 1, -1, 1, \dots, -1, 1}_{4g(K)-q_1-q_2-2\textup{ entries}}, -b'_n, -a'_n, \dots, -b'_1, -a'_1).\]

By reversing the roles of $q_1$ and $q_2$, we see that if $0 < q_1, q_2 < 4g(K)$, or if $q_1, q_2<0$, then $K_{2, q_1} \# T_{2, q_2}$ and $K_{2, q_2} \# T_{2, q_1}$ are locally equivalent, as desired.


For the ``only if'' direction, first consider the case where $q_1$ and $q_2$ have different signs. As mentioned in the introduction, applying the cabling formula~\cite{Hom:2014} for the $\tau$-invariant allows us to conclude that the $\tau$ of $P(K, U, 2, q_1, q_2)$ is nonzero. Thus, it suffices to consider the case where both $q_1$ and $q_2$ are positive. This case follows from comparing the previous computations, noting that $b'_n$ from Proposition~\ref{prop:cableL} equals $2b_n - 1$ if $q > 4g(K)$, and $2b_n$ if $q < 4g(K)$.

Lastly, we now assume that the odd integers $q_1$ and $q_2$ satisfy $q_1, q_2 < 0$, and consider the double branched cover of $P(K, U, 2, q_1, q_2)$, which is homeomorphic to
\[
\Sigma_2(P(K, U, 2, q_1, q_2)) := S^3_{q_1}(K \# K^r) \mathbin{\#} -S^3_{q_2}(K \# K^r) \# S^3_{q_2}(U) \mathbin{\#} -S^3_{q_1}(U),
\]
where $K^r$ denotes the reverse of $K$, $U$ is the unknot, and $S^3_m(J)$ denotes the $m$-framed Dehn surgery on a knot~$J$.
The Ozsv\'{a}th--Szab\'{o} $d$-invariant is additive under connected sum and behaves naturally under orientation reversal~\cite{Oz-Sz:dinvt}. Moreover, the $d$-invariant of a positive Dehn surgery is determined by the knot's $V_i$-sequence~\cite{Rasmussen:2004, Ni-Wu:2015}, and each $V_i$ vanishes when the $\nu^+$-invariant vanishes~\cite{Hom-Wu:2016}. It is also known that $\nu^+(-(K \# K^r)) = \nu^+(-(U))= 0$~(see e.g. \cite[Theorem 1.5]{BCG:2017}). 
In particular, for each negative integer $m$, there is a natural identification between the spin$^c$ structures of $S^3_m(K \# K^r)$ and those of $S^3_m(U)$, under which the corresponding $d$-invariants coincide. Hence, the same holds for the $d$-invariants of $\Sigma_2(P(K, U, 2, q_1, q_2))$ and those of
\[
S^3_{q_1}(U) \mathbin{\#} -S^3_{q_2}(U) \# S^3_{q_2}(U) \mathbin{\#} -S^3_{q_1}(U),
\]
which bounds a spin rational homology ball. This implies that the obstruction coming from the $d$-invariant vanishes. Furthermore, the vanishing of the obstruction given by the $\overline{d}$- and $\underline{d}$-invariants follows from~\cite[Theorem~1.6]{HHSZ:2020}, which states that the local equivalence class of the $\iota$-complex $(\CFm, \iota)$ associated to $S^3_{q_1}(J)$ coincides with that of $S^3_{q_2}(J)$, with respect to the unique spin structure, for any choice of knot~$J$, as long as the signs of $q_1$ and $q_2$ agree. This concludes the proof.
\end{proof}

\bibliographystyle{alpha}
\bibliography{bib}

\begin{bibdiv}
\begin{biblist}

\bib{Agol:2022}{article}{
      author={Agol, Ian},
       title={Ribbon concordance of knots is a partial ordering},
        date={2022},
        ISSN={2692-3688},
     journal={Commun. Am. Math. Soc.},
      volume={2},
       pages={374\ndash 379},
         url={https://doi.org/10.1090/cams/15},
      review={\MR{4520779}},
}

\bib{Abe-Tagami:2016}{article}{
      author={Abe, Tetsuya},
      author={Tagami, Keiji},
       title={Fibered knots with the same 0-surgery and the slice-ribbon
  conjecture},
        date={2016},
        ISSN={1073-2780,1945-001X},
     journal={Math. Res. Lett.},
      volume={23},
      number={2},
       pages={303\ndash 323},
         url={https://doi.org/10.4310/MRL.2016.v23.n2.a1},
      review={\MR{3512887}},
}

\bib{Baker:2016}{article}{
      author={Baker, Kenneth~L.},
       title={A note on the concordance of fibered knots},
        date={2016},
        ISSN={1753-8416,1753-8424},
     journal={J. Topol.},
      volume={9},
      number={1},
       pages={1\ndash 4},
         url={https://doi.org/10.1112/jtopol/jtv024},
      review={\MR{3465837}},
}

\bib{BCG:2017}{article}{
      author={Bodn\'ar, J\'ozsef},
      author={Celoria, Daniele},
      author={Golla, Marco},
       title={A note on cobordisms of algebraic knots},
        date={2017},
        ISSN={1472-2747,1472-2739},
     journal={Algebr. Geom. Topol.},
      volume={17},
      number={4},
       pages={2543\ndash 2564},
         url={https://doi.org/10.2140/agt.2017.17.2543},
      review={\MR{3686406}},
}

\bib{BCP:2018-1}{misc}{
      author={Borodzik, Maciej},
      author={Conway, Anthony},
      author={Politarczyk, Wojciech},
       title={Twisted {B}lanchfield pairings and twisted signatures {II}:
  {R}elation to {C}asson-{G}ordon invariants},
        date={2018},
         url={https://arxiv.org/abs/1809.08791},
}

\bib{BCP:2022-1}{article}{
      author={Borodzik, Maciej},
      author={Conway, Anthony},
      author={Politarczyk, Wojciech},
       title={Twisted {B}lanchfield pairings and twisted signatures {I}:
  {A}lgebraic background},
        date={2022},
        ISSN={0024-3795,1873-1856},
     journal={Linear Algebra Appl.},
      volume={655},
       pages={236\ndash 290},
         url={https://doi.org/10.1016/j.laa.2022.09.011},
      review={\MR{4493203}},
}

\bib{BCP:2024-3}{article}{
      author={Borodzik, Maciej},
      author={Conway, Anthony},
      author={Politarczyk, Wojciech},
       title={Twisted {B}lanchfield pairings and twisted signatures {III}:
  {A}pplications},
        date={2024},
        ISSN={0017-0895,1469-509X},
     journal={Glasg. Math. J.},
      volume={66},
      number={3},
       pages={501\ndash 540},
         url={https://doi.org/10.1017/S0017089524000077},
      review={\MR{4829281}},
}

\bib{Casson-Gordon:1983}{article}{
      author={Casson, Andrew~J.},
      author={Gordon, Cameron~McA.},
       title={A loop theorem for duality spaces and fibred ribbon knots},
        date={1983},
        ISSN={0020-9910,1432-1297},
     journal={Invent. Math.},
      volume={74},
      number={1},
       pages={119\ndash 137},
         url={https://doi.org/10.1007/BF01388533},
      review={\MR{722728}},
}

\bib{Casson-Gordon:1986}{incollection}{
      author={Casson, Andrew~J.},
      author={Gordon, Cameron~McA.},
       title={Cobordism of classical knots},
        date={1986},
   booktitle={\`a{} la recherche de la topologie perdue},
      series={Progr. Math.},
      volume={62},
   publisher={Birkh\"auser Boston, Boston, MA},
       pages={181\ndash 199},
        note={With an appendix by P. M. Gilmer},
      review={\MR{900252}},
}

\bib{ChenHanselman}{misc}{
      author={Chen, Wenzhao},
      author={Hanselman, Jonathan},
       title={Satellite knots and immersed {H}eegaard {F}loer homology},
        date={2023},
         url={https://arxiv.org/abs/2309.12297},
}

\bib{Conway-Kim-Politarczyk:2023}{article}{
      author={Conway, Anthony},
      author={Kim, Min~Hoon},
      author={Politarczyk, Wojciech},
       title={Nonslice linear combinations of iterated torus knots},
        date={2023},
        ISSN={1472-2747,1472-2739},
     journal={Algebr. Geom. Topol.},
      volume={23},
      number={2},
       pages={765\ndash 802},
         url={https://doi.org/10.2140/agt.2023.23.765},
      review={\MR{4587316}},
}

\bib{DHST-conc}{article}{
      author={Dai, Irving},
      author={Hom, Jennifer},
      author={Stoffregen, Matthew},
      author={Truong, Linh},
       title={More concordance homomorphisms from knot {F}loer homology},
        date={2021},
        ISSN={1465-3060,1364-0380},
     journal={Geom. Topol.},
      volume={25},
      number={1},
       pages={275\ndash 338},
         url={https://doi.org/10.2140/gt.2021.25.275},
      review={\MR{4226231}},
}

\bib{DHST-homconc}{article}{
      author={Dai, Irving},
      author={Hom, Jennifer},
      author={Stoffregen, Matthew},
      author={Truong, Linh},
       title={Homology concordance and knot {F}loer homology},
        date={2024},
        ISSN={0025-5831,1432-1807},
     journal={Math. Ann.},
      volume={390},
      number={4},
       pages={6111\ndash 6186},
         url={https://doi.org/10.1007/s00208-024-02906-9},
      review={\MR{4816130}},
}

\bib{Fox:1961}{incollection}{
      author={Fox, Ralph~H.},
       title={Some problems in knot theory},
        date={1961},
   booktitle={Topology of 3-manifolds and related topics ({P}roc. {T}he {U}niv.
  of {G}eorgia {I}nstitute, 1961)},
   publisher={Prentice-Hall, Inc., Englewood Cliffs, NJ},
       pages={168\ndash 176},
      review={\MR{140100}},
}

\bib{Ghiggini:2008}{article}{
      author={Ghiggini, Paolo},
       title={Knot {F}loer homology detects genus-one fibred knots},
        date={2008},
        ISSN={0002-9327,1080-6377},
     journal={Amer. J. Math.},
      volume={130},
      number={5},
       pages={1151\ndash 1169},
         url={https://doi.org/10.1353/ajm.0.0016},
      review={\MR{2450204}},
}

\bib{Gordon-Luecke:1989}{article}{
      author={Gordon, Cameron~McA.},
      author={Luecke, John},
       title={Knots are determined by their complements},
        date={1989},
        ISSN={0894-0347,1088-6834},
     journal={J. Amer. Math. Soc.},
      volume={2},
      number={2},
       pages={371\ndash 415},
         url={https://doi.org/10.2307/1990979},
      review={\MR{965210}},
}

\bib{Gordon:1981}{article}{
      author={Gordon, Cameron~McA.},
       title={Ribbon concordance of knots in the {$3$}-sphere},
        date={1981},
        ISSN={0025-5831,1432-1807},
     journal={Math. Ann.},
      volume={257},
      number={2},
       pages={157\ndash 170},
         url={https://doi.org/10.1007/BF01458281},
      review={\MR{634459}},
}

\bib{Hedden:2010}{article}{
      author={Hedden, Matthew},
       title={Notions of positivity and the {O}zsv\'ath-{S}zab\'o{} concordance
  invariant},
        date={2010},
        ISSN={0218-2165,1793-6527},
     journal={J. Knot Theory Ramifications},
      volume={19},
      number={5},
       pages={617\ndash 629},
         url={https://doi.org/10.1142/S0218216510008017},
      review={\MR{2646650}},
}

\bib{HHL}{article}{
      author={Hendricks, Kristen},
      author={Hom, Jennifer},
      author={Lidman, Tye},
       title={Applications of involutive {H}eegaard {F}loer homology},
        date={2021},
        ISSN={1474-7480,1475-3030},
     journal={J. Inst. Math. Jussieu},
      volume={20},
      number={1},
       pages={187\ndash 224},
         url={https://doi.org/10.1017/S147474801900015X},
      review={\MR{4205781}},
}

\bib{HHSZ:2020}{misc}{
      author={Hendricks, Kristen},
      author={Hom, Jennifer},
      author={Stoffregen, Matthew},
      author={Zemke, Ian},
       title={Surgery exact triangles in involutive heegaard floer homology},
        date={2020},
         url={https://arxiv.org/abs/2011.00113},
}

\bib{Hedden-Kirk-Livingston:2012}{article}{
      author={Hedden, Matthew},
      author={Kirk, Paul},
      author={Livingston, Charles},
       title={Non-slice linear combinations of algebraic knots},
        date={2012},
        ISSN={1435-9855,1435-9863},
     journal={J. Eur. Math. Soc. (JEMS)},
      volume={14},
      number={4},
       pages={1181\ndash 1208},
         url={https://doi.org/10.4171/JEMS/330},
      review={\MR{2928848}},
}

\bib{HKPS:2022}{article}{
      author={Hom, Jennifer},
      author={Kang, Sungkyung},
      author={Park, JungHwan},
      author={Stoffregen, Matthew},
       title={Linear independence of rationally slice knots},
        date={2022},
        ISSN={1465-3060,1364-0380},
     journal={Geom. Topol.},
      volume={26},
      number={7},
       pages={3143\ndash 3172},
         url={https://doi.org/10.2140/gt.2022.26.3143},
      review={\MR{4540903}},
}

\bib{HendricksManolescu}{article}{
      author={Hendricks, Kristen},
      author={Manolescu, Ciprian},
       title={Involutive {H}eegaard {F}loer homology},
        date={2017},
        ISSN={0012-7094},
     journal={Duke Math. J.},
      volume={166},
      number={7},
       pages={1211\ndash 1299},
         url={https://doi-org.prx.library.gatech.edu/10.1215/00127094-3793141},
      review={\MR{3649355}},
}

\bib{Hom:2014}{article}{
      author={Hom, Jennifer},
       title={Bordered {H}eegaard {F}loer homology and the tau-invariant of
  cable knots},
        date={2014},
        ISSN={1753-8416,1753-8424},
     journal={J. Topol.},
      volume={7},
      number={2},
       pages={287\ndash 326},
         url={https://doi.org/10.1112/jtopol/jtt030},
      review={\MR{3217622}},
}

\bib{Hom-CFKconc}{article}{
      author={Hom, Jennifer},
       title={The knot {F}loer complex and the smooth concordance group},
        date={2014},
        ISSN={0010-2571,1420-8946},
     journal={Comment. Math. Helv.},
      volume={89},
      number={3},
       pages={537\ndash 570},
         url={https://doi.org/10.4171/CMH/326},
      review={\MR{3260841}},
}

\bib{Hom-survey}{article}{
      author={Hom, Jennifer},
       title={A survey on {H}eegaard {F}loer homology and concordance},
        date={2017},
        ISSN={0218-2165,1793-6527},
     journal={J. Knot Theory Ramifications},
      volume={26},
      number={2},
       pages={1740015, 24},
         url={https://doi.org/10.1142/S0218216517400156},
      review={\MR{3604497}},
}

\bib{Hom-lecture}{incollection}{
      author={Hom, Jennifer},
       title={Lecture notes on {H}eegaard {F}loer homology},
        date={2021},
   booktitle={Quantum field theory and manifold invariants},
      series={IAS/Park City Math. Ser.},
      volume={28},
   publisher={Amer. Math. Soc., Providence, RI},
       pages={171\ndash 200},
      review={\MR{4380679}},
}

\bib{HRW:2022}{article}{
      author={Hanselman, Jonathan},
      author={Rasmussen, Jacob},
      author={Watson, Liam},
       title={Heegaard {F}loer homology for manifolds with torus boundary:
  properties and examples},
        date={2022},
        ISSN={0024-6115,1460-244X},
     journal={Proc. Lond. Math. Soc. (3)},
      volume={125},
      number={4},
       pages={879\ndash 967},
         url={https://doi.org/10.1112/plms.12473},
      review={\MR{4500201}},
}

\bib{HRW:2024}{article}{
      author={Hanselman, Jonathan},
      author={Rasmussen, Jacob},
      author={Watson, Liam},
       title={Bordered {F}loer homology for manifolds with torus boundary via
  immersed curves},
        date={2024},
        ISSN={0894-0347,1088-6834},
     journal={J. Amer. Math. Soc.},
      volume={37},
      number={2},
       pages={391\ndash 498},
         url={https://doi.org/10.1090/jams/1029},
      review={\MR{4695506}},
}

\bib{Hom-Wu:2016}{article}{
      author={Hom, Jennifer},
      author={Wu, Zhongtao},
       title={Four-ball genus bounds and a refinement of the
  {O}zsv\'ath-{S}zab\'o{} tau invariant},
        date={2016},
        ISSN={1527-5256,1540-2347},
     journal={J. Symplectic Geom.},
      volume={14},
      number={1},
       pages={305\ndash 323},
         url={https://doi.org/10.4310/JSG.2016.v14.n1.a12},
      review={\MR{3523259}},
}

\bib{Hanselman-Watson:2023}{article}{
      author={Hanselman, Jonathan},
      author={Watson, Liam},
       title={Cabling in terms of immersed curves},
        date={2023},
        ISSN={1465-3060,1364-0380},
     journal={Geom. Topol.},
      volume={27},
      number={3},
       pages={925\ndash 952},
         url={https://doi.org/10.2140/gt.2023.27.925},
      review={\MR{4599309}},
}

\bib{Johannson:1979}{book}{
      author={Johannson, Klaus},
       title={Homotopy equivalences of {$3$}-manifolds with boundaries},
      series={Lecture Notes in Mathematics},
   publisher={Springer, Berlin},
        date={1979},
      volume={761},
        ISBN={3-540-09714-7},
      review={\MR{551744}},
}

\bib{Jaco-Shalen:1979}{incollection}{
      author={Jaco, William},
      author={Shalen, Peter~B.},
       title={A new decomposition theorem for irreducible sufficiently-large
  {$3$}-manifolds},
        date={1978},
   booktitle={Algebraic and geometric topology ({P}roc. {S}ympos. {P}ure
  {M}ath., {S}tanford {U}niv., {S}tanford, {C}alif., 1976), {P}art 2},
      series={Proc. Sympos. Pure Math.},
      volume={XXXII},
   publisher={Amer. Math. Soc., Providence, RI},
       pages={71\ndash 84},
      review={\MR{520524}},
}

\bib{Kearton:1979}{article}{
      author={Kearton, C.},
       title={The {M}ilnor signatures of compound knots},
        date={1979},
        ISSN={0002-9939,1088-6826},
     journal={Proc. Amer. Math. Soc.},
      volume={76},
      number={1},
       pages={157\ndash 160},
         url={https://doi.org/10.2307/2042936},
      review={\MR{534409}},
}

\bib{Kochloukova:2006}{article}{
      author={Kochloukova, Dessislava~H.},
       title={Some {N}ovikov rings that are von {N}eumann finite and knot-like
  groups},
        date={2006},
        ISSN={0010-2571,1420-8946},
     journal={Comment. Math. Helv.},
      volume={81},
      number={4},
       pages={931\ndash 943},
         url={https://doi.org/10.4171/CMH/81},
      review={\MR{2271229}},
}

\bib{Livingston:2004}{article}{
      author={Livingston, Charles},
       title={Computations of the {O}zsv\'ath-{S}zab\'o{} knot concordance
  invariant},
        date={2004},
        ISSN={1465-3060,1364-0380},
     journal={Geom. Topol.},
      volume={8},
       pages={735\ndash 742},
         url={https://doi.org/10.2140/gt.2004.8.735},
      review={\MR{2057779}},
}

\bib{Livingston-Melvin:1983}{article}{
      author={Livingston, Charles},
      author={Melvin, Paul},
       title={Algebraic knots are algebraically dependent},
        date={1983},
        ISSN={0002-9939,1088-6826},
     journal={Proc. Amer. Math. Soc.},
      volume={87},
      number={1},
       pages={179\ndash 180},
         url={https://doi.org/10.2307/2044377},
      review={\MR{677257}},
}

\bib{LOT}{article}{
      author={Lipshitz, Robert},
      author={Ozsvath, Peter},
      author={Thurston, Dylan~P.},
       title={Bordered {H}eegaard {F}loer homology},
        date={2018},
        ISSN={0065-9266,1947-6221},
     journal={Mem. Amer. Math. Soc.},
      volume={254},
      number={1216},
       pages={viii+279},
         url={https://doi.org/10.1090/memo/1216},
      review={\MR{3827056}},
}

\bib{Miyazaki:2018}{article}{
      author={Miyazaki, Katura},
       title={A note on genera of band sums that are fibered},
        date={2018},
        ISSN={0218-2165,1793-6527},
     journal={J. Knot Theory Ramifications},
      volume={27},
      number={12},
       pages={1871002, 3},
         url={https://doi.org/10.1142/S0218216518710025},
      review={\MR{3876350}},
}

\bib{Miyazaki:1994}{article}{
      author={Miyazaki, Katura},
       title={Nonsimple, ribbon fibered knots},
        date={1994},
        ISSN={0002-9947,1088-6850},
     journal={Trans. Amer. Math. Soc.},
      volume={341},
      number={1},
       pages={1\ndash 44},
         url={https://doi.org/10.2307/2154613},
      review={\MR{1176509}},
}

\bib{Miller-Powell:2018}{article}{
      author={Miller, Allison},
      author={Powell, Mark},
       title={Symmetric chain complexes, twisted {B}lanchfield pairings and
  knot concordance},
        date={2018},
        ISSN={1472-2747,1472-2739},
     journal={Algebr. Geom. Topol.},
      volume={18},
      number={6},
       pages={3425\ndash 3476},
         url={https://doi.org/10.2140/agt.2018.18.3425},
      review={\MR{3868226}},
}

\bib{Ni:2007}{article}{
      author={Ni, Yi},
       title={Knot {F}loer homology detects fibred knots},
        date={2007},
        ISSN={0020-9910,1432-1297},
     journal={Invent. Math.},
      volume={170},
      number={3},
       pages={577\ndash 608},
         url={https://doi.org/10.1007/s00222-007-0075-9},
      review={\MR{2357503}},
}

\bib{Ni-Wu:2015}{article}{
      author={Ni, Yi},
      author={Wu, Zhongtao},
       title={Cosmetic surgeries on knots in {$S^3$}},
        date={2015},
        ISSN={0075-4102,1435-5345},
     journal={J. Reine Angew. Math.},
      volume={706},
       pages={1\ndash 17},
         url={https://doi.org/10.1515/crelle-2013-0067},
      review={\MR{3393360}},
}

\bib{Oz-Sz:dinvt}{article}{
      author={Ozsv\'ath, Peter},
      author={Szab\'o, Zolt\'an},
       title={Absolutely graded {F}loer homologies and intersection forms for
  four-manifolds with boundary},
        date={2003},
        ISSN={0001-8708,1090-2082},
     journal={Adv. Math.},
      volume={173},
      number={2},
       pages={179\ndash 261},
         url={https://doi.org/10.1016/S0001-8708(02)00030-0},
      review={\MR{1957829}},
}

\bib{Oz-Sz:2003}{article}{
      author={Ozsv\'ath, Peter},
      author={Szab\'o, Zolt\'an},
       title={Knot {F}loer homology and the four-ball genus},
        date={2003},
        ISSN={1465-3060,1364-0380},
     journal={Geom. Topol.},
      volume={7},
       pages={615\ndash 639},
         url={https://doi.org/10.2140/gt.2003.7.615},
      review={\MR{2026543}},
}

\bib{OS-genus}{article}{
      author={Ozsv\'ath, Peter},
      author={Szab\'o, Zolt\'an},
       title={Holomorphic disks and genus bounds},
        date={2004},
        ISSN={1465-3060,1364-0380},
     journal={Geom. Topol.},
      volume={8},
       pages={311\ndash 334},
         url={https://doi.org/10.2140/gt.2004.8.311},
      review={\MR{2023281}},
}

\bib{OS-knots}{article}{
      author={Ozsv\'ath, Peter},
      author={Szab\'o, Zolt\'an},
       title={Holomorphic disks and knot invariants},
        date={2004},
        ISSN={0001-8708,1090-2082},
     journal={Adv. Math.},
      volume={186},
      number={1},
       pages={58\ndash 116},
         url={https://doi.org/10.1016/j.aim.2003.05.001},
      review={\MR{2065507}},
}

\bib{OS:fibered}{article}{
      author={Ozsv\'ath, Peter},
      author={Szab\'o, Zolt\'an},
       title={Holomorphic triangle invariants and the topology of symplectic
  four-manifolds},
        date={2004},
        ISSN={0012-7094,1547-7398},
     journal={Duke Math. J.},
      volume={121},
      number={1},
       pages={1\ndash 34},
         url={https://doi.org/10.1215/S0012-7094-04-12111-6},
      review={\MR{2031164}},
}

\bib{OS:lensspace}{article}{
      author={Ozsv\'ath, Peter},
      author={Szab\'o, Zolt\'an},
       title={On knot {F}loer homology and lens space surgeries},
        date={2005},
        ISSN={0040-9383},
     journal={Topology},
      volume={44},
      number={6},
       pages={1281\ndash 1300},
         url={https://doi.org/10.1016/j.top.2005.05.001},
      review={\MR{2168576}},
}

\bib{OS-intro}{incollection}{
      author={Ozsv\'ath, Peter},
      author={Szab\'o, Zolt\'an},
       title={An introduction to {H}eegaard {F}loer homology},
        date={2006},
   booktitle={Floer homology, gauge theory, and low-dimensional topology},
      series={Clay Math. Proc.},
      volume={5},
   publisher={Amer. Math. Soc., Providence, RI},
       pages={3\ndash 27},
      review={\MR{2249247}},
}

\bib{OSS:2017}{article}{
      author={Ozsv\'ath, Peter},
      author={Stipsicz, Andr\'as~I.},
      author={Szab\'o, Zolt\'an},
       title={Concordance homomorphisms from knot {F}loer homology},
        date={2017},
        ISSN={0001-8708,1090-2082},
     journal={Adv. Math.},
      volume={315},
       pages={366\ndash 426},
         url={https://doi.org/10.1016/j.aim.2017.05.017},
      review={\MR{3667589}},
}

\bib{Perelman:2002-1}{misc}{
      author={Perelman, Grisha},
       title={The entropy formula for the {R}icci flow and its geometric
  applications},
        date={2002},
         url={https://arxiv.org/abs/0211159},
}

\bib{Perelman:2003-1}{misc}{
      author={Perelman, Grisha},
       title={Finite extinction time for the solutions to the {R}icci flow on
  certain three-manifolds},
        date={2003},
         url={https://arxiv.org/abs/0307245},
}

\bib{Perelman:2003-2}{misc}{
      author={Perelman, Grisha},
       title={Ricci flow with surgery on three-manifolds},
        date={2003},
         url={https://arxiv.org/abs/0303109},
}

\bib{Rasmussen-thesis}{book}{
      author={Rasmussen, Jacob},
       title={Floer homology and knot complements},
   publisher={ProQuest LLC, Ann Arbor, MI},
        date={2003},
        ISBN={978-0496-39374-9},
  url={http://gateway.proquest.com.prx.library.gatech.edu/openurl?url_ver=Z39.88-2004&rft_val_fmt=info:ofi/fmt:kev:mtx:dissertation&res_dat=xri:pqdiss&rft_dat=xri:pqdiss:3091665},
        note={Thesis (Ph.D.)--Harvard University},
      review={\MR{2704683}},
}

\bib{Rasmussen:2004}{article}{
      author={Rasmussen, Jacob},
       title={Lens space surgeries and a conjecture of {G}oda and {T}eragaito},
        date={2004},
        ISSN={1465-3060,1364-0380},
     journal={Geom. Topol.},
      volume={8},
       pages={1013\ndash 1031},
         url={https://doi.org/10.2140/gt.2004.8.1013},
      review={\MR{2087076}},
}

\bib{Rasmussen:2010}{article}{
      author={Rasmussen, Jacob},
       title={Khovanov homology and the slice genus},
        date={2010},
        ISSN={0020-9910,1432-1297},
     journal={Invent. Math.},
      volume={182},
      number={2},
       pages={419\ndash 447},
         url={https://doi.org/10.1007/s00222-010-0275-6},
      review={\MR{2729272}},
}

\bib{Rudolph:1976}{article}{
      author={Rudolph, Lee},
       title={How independent are the knot-cobordism classes of links of plane
  curve singularities?},
        date={1976},
     journal={Notices Amer. Math. Soc.},
      volume={23},
      number={410},
}

\bib{Silver:1992}{article}{
      author={Silver, D.~S.},
       title={On knot-like groups and ribbon concordance},
        date={1992},
        ISSN={0022-4049,1873-1376},
     journal={J. Pure Appl. Algebra},
      volume={82},
      number={1},
       pages={99\ndash 105},
         url={https://doi.org/10.1016/0022-4049(92)90013-6},
      review={\MR{1181096}},
}

\bib{Saito-Teragaito:2010}{article}{
      author={Saito, Toshio},
      author={Teragaito, Masakazu},
       title={Knots yielding homeomorphic lens spaces by {D}ehn surgery},
        date={2010},
        ISSN={0030-8730,1945-5844},
     journal={Pacific J. Math.},
      volume={244},
      number={1},
       pages={169\ndash 192},
         url={https://doi.org/10.2140/pjm.2010.244.169},
      review={\MR{2580533}},
}

\bib{Stoffregen-Pin2}{article}{
      author={Stoffregen, Matthew},
       title={Pin(2)-equivariant {S}eiberg-{W}itten {F}loer homology of
  {S}eifert fibrations},
        date={2020},
        ISSN={0010-437X,1570-5846},
     journal={Compos. Math.},
      volume={156},
      number={2},
       pages={199\ndash 250},
         url={https://doi.org/10.1112/s0010437x19007620},
      review={\MR{4044465}},
}

\bib{Zemke-connected}{article}{
      author={Zemke, Ian},
       title={Connected sums and involutive knot {F}loer homology},
        date={2019},
        ISSN={0024-6115,1460-244X},
     journal={Proc. Lond. Math. Soc. (3)},
      volume={119},
      number={1},
       pages={214\ndash 265},
         url={https://doi.org/10.1112/plms.12227},
      review={\MR{3957835}},
}

\bib{Zemke-link}{article}{
      author={Zemke, Ian},
       title={Link cobordisms and functoriality in link {F}loer homology},
        date={2019},
        ISSN={1753-8416,1753-8424},
     journal={J. Topol.},
      volume={12},
      number={1},
       pages={94\ndash 220},
         url={https://doi.org/10.1112/topo.12085},
      review={\MR{3905679}},
}

\bib{Zhang}{misc}{
      author={Zhang, Chen},
       title={D invariants obstruction to sliceness of a class of algebraically
  slice knots},
        date={2023},
         url={https://arxiv.org/abs/2311.06944},
}

\bib{Zhou-homconc}{article}{
      author={Zhou, Hugo},
       title={Homology concordance and an infinite rank free subgroup},
        date={2021},
        ISSN={1753-8416,1753-8424},
     journal={J. Topol.},
      volume={14},
      number={4},
       pages={1369\ndash 1395},
         url={https://doi.org/10.1112/topo.12211},
      review={\MR{4406694}},
}

\end{biblist}
\end{bibdiv}

\end{document}